\documentclass[11pt]{article}      
\usepackage{geometry}                       
\usepackage{graphicx}               
\usepackage{amsfonts}
\usepackage{mathrsfs}
\usepackage{amsmath}
\usepackage{amsthm}                             
\usepackage{amssymb}
\usepackage{cite}
\usepackage{xcolor}
\usepackage{enumerate}
\makeatletter
\@addtoreset{equation}{section}
\makeatother

\def\ds{\displaystyle}
\def\={\buildrel \triangle \over =}
\def\ns{\noalign{\ss} }

\def\ms{\medskip}

\def\q{\quad}
\def\qq{\qquad}

\def\3n{\negthinspace \negthinspace \negthinspace }
\def\2n{\negthinspace \negthinspace }
\def\1n{\negthinspace }

\def\|{||}
\def\({\Big (}
\def\){\Big )}
\def\[{\Big[}
\def\]{\Big]}
\def\be{\begin{equation}}
\def\bel{\begin{equation}\label}
\def\ee{\end{equation}}
\def\bt{\begin{theorem}}
\def\bcd{\begin{condition}}
\def\ecd{\end{condition}}
\def\et{\end{theorem}}
\def\bc{\begin{corollary}}
\def\ec{\end{corollary}}
\def\bde{\begin{definition}}
\def\ede{\end{definition}}
\def\bl{\begin{lemma}}
\def\el{\end{lemma}}
\def\bp{\begin{proposition}}
\def\ep{\end{proposition}}
\def\br{\begin{remark}}
\def\er{\end{remark}}
\def\ba{\begin{array}}
\def\ea{\end{array}}
\def\ed{\end{document}}
\def\ns{\noalign{\ms}}
\def\ds{\displaystyle}

\newtheorem{lemma}{Lemma}[section]
\newtheorem{remark}{Remark}[section]

\newtheorem{theorem}{Theorem}[section]
\newtheorem{corollary}{Corollary}[section]

\newtheorem{definition}{Definition}[section]
\newtheorem{proposition}{Proposition}[section]
\newtheorem{condition}{Condition}[section]

\allowdisplaybreaks

\title{\bf Conditional Stability of Coefficients Inverse
Problem for Strongly Coupled Schr\"{o}dinger Equations}
\author{Xiaomin Zhu\thanks{School of Mathematical Sciences, University of Electronic Science and Technology of China}   \ and Fangfang Dou\thanks{Corresponding author. School of Mathematical Sciences, University of Electronic Science
and Technology of China. E-mail: fangfdou@uestc.edu.cn.  }}
\date{}                            

\begin{document}

\maketitle
\begin{abstract}
This paper concerns inverse problems for strongly coupled $\textrm{Schr\"{o}dinger} $ equations. The purpose of this inverse problem is to  retrieve a stationary potential in the strongly coupled $\textrm{Schr\"{o}dinger} $ equations from either boundary or internal measurements. Two stability results are derived from a new Carleman estimate for the strongly coupled $\textrm{Schr\"{o}dinger} $ equations.
\end{abstract}


\section{Introduction}
Let $T>0$, $i=\sqrt{-1}$ and $\Omega\subset\mathbb{R}^N$ be a bounded domain with smooth boundary $\Gamma$, and $\nu$ denotes the unit outward normal vector to $\Gamma$. The system of strongly coupled  $\textrm{Schr\"{o}dinger} $ equations
\begin{equation} \label{a1}
  \left\{
   \begin{array}{ll}
\ds i\partial_t y_1+a_{11}\Delta y_1+a_{12}\Delta y_2+a(x)y_1+b(x)y_2=f_1,\quad &\textrm{in}\ \Omega\times(0,T),  \\
\ns\ds i\partial_t y_2+a_{21}\Delta y_1+a_{22}\Delta y_2+c(x)y_1+d(x)y_2=f_2,\quad   &\textrm{in}\ \Omega\times(0,T),   \\
\ns\ds y_1(x,t)=g_1(x,t),y_2(x,t)=g_2(x,t),\quad  &\textrm{on}\ \Gamma\times(0,T), \\
\ns\ds y_1(x,0)=y_{10}, y_2(x,0)=y_{20}, \quad&\textrm{in}\ \Omega.
   \end{array}
   \right.
  \end{equation}
can be used to describe the motion of varies microscopic particles under strong perturbations. Coefficient identification of this system has been concerned by many research fields. In case of consideration for describing molecular multiphoton transitions induced by a laser, $a(x)$ and $d(x)$ in \eqref{a1} indicating the field-free molecular electronic potential, and $b(x)$ and $c(x)$ indicating the radiation-molecule interaction\cite{BA1993}. However, limited by the uncertainty principle, it is difficult to directly measure the field-free molecular electronic potential in the experiment. Thus, the problem of identifying the coefficients $a(x)$ and $d(x)$ in system \eqref{a1} is proposed in the real applications.

 Our purpose in this paper is retrieving a stationary potential in the strongly coupled $\textrm{Schr\"{o}dinger} $ equations from either boundary or internal measurements, which can be stated as follows.

\noindent{\bf Inverse Problem I (internal observations).}  Let $\omega$ be any given open subdomain of $\Omega$ and set $\omega_T:=\omega\times (0,T)$, our goal is to retrieve the potential $a(x)$ in the strongly coupled $\textrm{Schr\"{o}dinger} $ equations from the observation data
$$(y_1,y_2)|_{\omega_T}.$$
{\bf Inverse Problem II (boundary observations).}  Let $\Gamma^+$ be an open set of $\Gamma$, under some geometrically conditions on $\Gamma^+$, and set $\Gamma_T^+:=\Gamma^+\times (0,T)$, our task is to retrieve the potential $a(x)$ in the strongly coupled $\textrm{Schr\"{o}dinger} $ equations from the measurement of the normal derivative
$$(\frac{\partial y_1}{\partial\nu},\frac{\partial y_2}{\partial\nu})\big|_{\Gamma^+_T}.$$

The main tool we used in this problem is a new Carleman estimate for strongly coupled $\textrm{Schr\"{o}dinger} $ equations. Carleman estimate which is introduced by Carleman\cite{CT1939} is a weighted energy estimate. Bukhgeim and Klibanov introduced the Carleman estimate into the field of inverse problems in the paper of \cite{BK1981}(see also\cite{KB1992} and \cite{KB2013}). The Carleman estimate of $\textrm{Schr\"{o}dinger} $ equations were studied extensively(e.g.\cite{BM2009} \cite{LB2009} \cite{BM2008}). Isakov\cite{VI1998} and Tataru\cite{TD1997} proved local Carleman estimates for $\textrm{Schr\"{o}dinger} $ equation under a strong  pseudoconvexity condition while Baudouin and Puel\cite{BP2007} established a global Carleman estimate for the evolution $\textrm{Schr\"{o}dinger} $ equation under a strict pseudoconvexity condition and use it to derive the stability for some inverse problems. The idea of relaxing the pseudoconvexity condition into a weak one was proposed in \cite{KJ2007}, after that, Osses et al.\cite{O2008} established some new Carleman inequalities under a relaxed pseudoconvexity condition and applied them to obtain the stability of the inverse problem for the evolution $\textrm{Schr\"{o}dinger} $ equation. For the case of coupled systems, Benabdallah et al.\cite{BC2009}, Cristofoletal et al.\cite{CP2006} and Dou and Yamamoto\cite{DY2019} derived the Carleman estimate for some weakly coupled systems by adding the Carleman estimate corresponding to each equation. Liu and Triggiani\cite{LT2011} proved the uniqueness theorem in determining electric potentials for a system of coupled $\textrm{Schr\"{o}dinger} $ equations. However, since the strongly coupled terms cannot be absorbed by adding up each Carleman estimate at the same time, for strongly coupled systems, the method above doesn't work. In this case, if the coefficient matrix of the strongly coupled parabolic or hyperbolic system is diagonalizable, Bellassoued and Yamamoto\cite{BY2013} obtained the Carleman estimate of the strongly coupled system by diagonalizing the coefficient matrix. And if the coefficient matrix of the strongly coupled reaction-diffusion system is nonsingular, Wu and Yu\cite{WY2018} established a Carleman estimate and applied it to obtain the stability of coefficient inverse problem. To the best of our knowledge, there is no stability results for strongly coupled $\textrm{Schr\"{o}dinger} $ equations.

In this paper, we derive a new Carleman estimate for the strongly coupled $\textrm{Schr\"{o}dinger} $ equations. By converting the strongly coupled terms to the derivatives of $y_1$ and $y_2$ with respect to time, we establish the Carleman estimate for the system \eqref{a1}. For IP (I) and (II), the proof of the stability is based on the Carleman estimate for the strongly coupled $\textrm{Schr\"{o}dinger} $ equations.

This paper is organized as follows. In section 2, we derive a global Carleman estimate for internal measurements in $\omega\subset\Omega$, then we applied the Carleman inequality to obtain the stability of our inverse problem. Section 3 gives the same analysis for boundary observations on $\Gamma^+\subset\Gamma$.

 \section{ IP (I): Internal Observations}
 \subsection{Carleman estimate}
  Let $\omega$ be any given open subdomain of $\Omega$, we set $Q_T:=\Omega\times(0,T)$, $\Gamma_T:=\Gamma\times(0,T)$ and $\omega_T:=\omega\times(0,T)$. Suppose $\psi(x)\in C^4(\overline{\Omega})$ is a weight function which satisfied the following properties:
 \begin{itemize}
      \item  $\nabla\psi\neq0,\quad\quad \textrm{in}\ \overline{\Omega\setminus\omega}$.
      \item  $\frac{\partial\psi}{\partial \nu}\leq 0,\qquad \ \textrm{on}\ \Gamma$.
      \item  There exist a constant $\mu>0$ such that $\forall x\in \overline{\Omega\backslash\omega}$ and $\forall \xi=(\xi_1,...,\xi_N)\in\mathbb{R}^N$,
      \begin{equation}
      |\nabla\psi\cdot\xi|^2\geq\mu|\xi|^2.
      \end{equation}
      \item  $\psi(x)>\frac{3}{4}\|\psi\|_{L^\infty(\Omega)},  \quad\forall x\in\Omega$.
 \end{itemize}
Set $C_\psi=\frac{3}{2}\|\psi\|_{L^\infty(\Omega)}$ and let
\begin{equation}
\theta(x,t)=\frac{e^{\lambda\psi(x)}}{t(T-t)},\ \varphi(x,t)=\frac{e^{\lambda C_\psi}-e^{\lambda\psi(x)}}{t(T-t)},\qq\forall (x,t)\in Q_T,
\end{equation}
with a positive parameter $\lambda$.

Firstly, we consider the following system of strongly coupled  $\textrm{Schr\"{o}dinger} $ equations:
 \begin{equation} \label{b1}
  \left\{
  \begin{array}{ll}
\ds i\partial_t y_1+a_{11}(x)\Delta y_1+a_{12}(x)\Delta y_2=f_1,\quad &\textrm{in}\ \Omega\times(0,T),  \\
\ns\ds i\partial_t y_2+a_{21}(x)\Delta y_1+a_{22}(x)\Delta y_2=f_2,\quad   &\textrm{in}\ \Omega\times(0,T),    \\
\ns\ds y_1(x,t)=0,y_2(x,t)=0,\quad  &\textrm{on}\ \Gamma\times(0,T), \\
\ns\ds y_1(x,0)=y_{10}, y_2(x,0)=y_{20}, \quad&\textrm{in}\ \Omega.
   \end{array}
   \right.
\end{equation}
 Assuming that $a_{ij}\in C^2(\overline{\Omega})$ such that
\begin{equation} \label{ap}
 a_{12}(x)a_{21}(x)>0,\ \textrm{det}(a_{ij})\neq0,\ a_{22}(x)\cdot\textrm{det}(a_{ij})>0,\qq x\in\overline{\Omega}.
 \end{equation}
Then we have the following result.

\bp\label{A}
Suppose there exist $\psi$, $\varphi$ and $\theta$ satisfy the above conditions, let $f_1$, $f_2\in L^2(\overline{Q_T})$, and $(y_1,y_2)\in [H^1(0,T;L^2(\Omega))\cap L^2(0,T;H^2(\Omega))]^2$ is the solution of \eqref{b1}. Then there exist constants $s_0\geq1$, $\lambda_0\geq1$ and $C>0$ such that for all $s>s_0$, $\lambda>\lambda_0$, the next inequality holds:
\begin{align}\label{e5}\nonumber
&\q  \int_{Q_T}\left(|\tilde{M}_{11}(y_1 ,y_2)|^2+| \tilde{M}_{12}(y_1 ,y_2)|^2+| \tilde{M}_{21}(y_1 ,y_2)|^2+| \tilde{M}_{22}(y_1 ,y_2)|^2\right)e^{-2s\varphi}dxdt\\\nonumber
&\q  +\int_{Q_T}\left[s^3\lambda^4\theta^3\big(|y_1|^2+|y_2|^2\big)+s\lambda^2\theta\big(|\nabla y_1|^2+|\nabla y_2|^2\big)\right]e^{-2s\varphi}dxdt\\\nonumber
&\q  +\int_{\Gamma_T}s\lambda\theta\left|\frac{\partial\psi}{\partial \nu}\right|\left(\left|\frac{\partial y_1}{\partial \nu}\right|^2+\left|\frac{\partial y_2}{\partial \nu}\right|^2\right)e^{-2s\varphi}dSdt\\\nonumber
&\leq C\int_{Q_T}\left(|f_1|^2+|f_2|^2\right)e^{-2s\varphi}dxdt + C\int_{\omega_T}s^3\lambda^4\theta^3\big(|y_1|^2+|y_2|^2\big)e^{-2s\varphi}dxdt\\
&\q+ C\int_{\omega_T}s\lambda^2\theta\big(|\nabla y_1|^2+|\nabla y_2|^2\big)e^{-2s\varphi}dxdt,
\end{align}
where $\tilde{M}_{11}$, $\tilde{M}_{12}$, $\tilde{M}_{21}$ and $\tilde{M}_{22}$ denote the following operators
\begin{equation}\label{h1}
\left\{
\begin{aligned}
\tilde{M}_{11}(y_1 ,y_2)&=2s\nabla\varphi\cdot\nabla y_1+(isb_{11}\varphi_t+s\Delta\varphi-2s^2|\nabla\varphi|^2)y_1+isb_{12}\varphi_ty_2,\\
\tilde{M}_{12}(y_1 ,y_2)&=(2s^2|\nabla\varphi|^2-isb_{11}\varphi_t-s\Delta\varphi)y_1-isb_{12}\varphi_ty_2+ib_{11}\partial_ty_1+ib_{12}\partial_ty_2\q\\
&\q-2s\nabla\varphi\cdot\nabla y_1+\Delta y_1,
\end{aligned}
   \right.
\end{equation}
\begin{equation}\label{h2}
\left\{
\begin{aligned}
\tilde{M}_{21}(y_1 ,y_2)&=2s\nabla\varphi\cdot\nabla y_2+(isb_{22}\varphi_t+s\Delta\varphi-2s^2|\nabla\varphi|^2)y_2+isb_{21}\varphi_ty_1,\\
\tilde{M}_{22}(y_1 ,y_2)&=-isb_{21}\varphi_ty_1+(2s^2|\nabla\varphi|^2-isb_{22}\varphi_t-s\Delta\varphi)y_2+ib_{21}\partial_ty_1+ib_{22}\partial_ty_2\\
&\q-2s\nabla\varphi\cdot\nabla y_2+\Delta y_2.
\end{aligned}
   \right.
\end{equation}
\ep
\begin{proof}
Since $\textrm{det}(a_{ij})\neq0$, the system \eqref{b1} can be rewritten as
 \begin{equation}\label{b2}
\left\{
\begin{array}{ll}
\ds ib_{11}(x)\partial_t y_1+ib_{12}(x)\partial_t y_2+\Delta y_1=F_1,\quad &\textrm{in}\ \Omega\times(0,T),  \\
\ns\ds ib_{21}(x)\partial_t y_1+ib_{22}(x)\partial_t y_2+\Delta y_2=F_2,\quad &\textrm{in}\ \Omega\times(0,T),   \\
\ns\ds y_1(x,t)=0,y_2(x,t)=0,\quad  &\textrm{on}\ \Gamma\times(0,T), \\
\ns\ds y_1(x,0)=y_{10}, y_2(x,0)=y_{20}, \quad&\textrm{in}\ \Omega.
 \end{array}
   \right.
\end{equation}
where
\begin{equation}\label{b3}
b_{11}=\frac{a_{22}}{\textrm{det}(a_{ij})},\ b_{12}=-\frac{a_{12}}{\textrm{det}(a_{ij})},\ b_{21}=-\frac{a_{21}}{\textrm{det}(a_{ij})},\ b_{22}=\frac{a_{11}}{\textrm{det}(a_{ij})}.
\end{equation}
\begin{equation}\label{b4}
F_1=\frac{a_{22}f_1-a_{12}f_2}{\textrm{det}(a_{ij})},\quad F_2=\frac{a_{11}f_2-a_{21}f_1}{\textrm{det}(a_{ij})}.
\end{equation}
According to the assumption \eqref{ap}, it follows that $b_{ij}\in C^2(\overline{\Omega})$, and
$$\textrm{det}(b_{ij})=\frac{1}{\textrm{det}(a_{ij})}\neq0,$$
and there exist a $\sigma_0>0$ such that $\sigma(x)>\sigma_0$ and
$$\sigma^2(x)b_{12}(x)=b_{21}(x).$$
Assume
\begin{equation}\label{b6}
u(x,t)=e^{-s\varphi(x,t)}y_1(x,t),\quad v(x,t)=e^{-s\varphi(x,t)}y_2(x,t).
\end{equation}
Substituting \eqref{b6} into the equations in system \eqref{b1}, and let
\begin{equation}
\begin{cases}\ds
M_{11}(u,v)=2s\nabla\varphi\cdot\nabla u+s\Delta\varphi u+i(sb_{11}\varphi_tu+sb_{12}\varphi_tv),\\
\ns\ds M_{12}(u,v)=i(b_{11}u_t+b_{12}v_t)+\Delta u+s^2|\nabla \varphi|^2u,
\end{cases}
\end{equation}
\begin{equation}
\begin{cases}\ds
M_{21}(u,v)=2s\nabla\varphi\cdot\nabla v+s\Delta\varphi v+i(sb_{21}\varphi_tu+sb_{22}\varphi_tv),\\
\ns\ds M_{22}(u,v)=i(b_{21}u_t+b_{22}v_t)+\Delta v+s^2|\nabla \varphi|^2v,
\end{cases}
\end{equation}
and
\begin{equation}\label{b11}
G_1=e^{-s\varphi}F_1,\quad G_2=e^{-s\varphi}F_2,
\end{equation}
then we have
\begin{equation}\label{a11-10}
\begin{array}{ll}\ds
\|\sigma M_{11}(u,v)\|+\|\sigma M_{12}(u,v)\|+\| M_{21}(u,v)\|+\| M_{22}(u,v)\| \\
\ns\ds\quad+2\text{Re}\left(\sigma M_{11}(u,v),\sigma M_{12}(u,v)\right)+2\text{Re}\left( M_{21}(u,v),M_{22}(u,v)\right)
\\
\ns\ds=\|\sigma G_1\|+\| G_2\|.
\end{array}
\end{equation}
Here $\| u\|$ is the standard norm in $L^2(Q_T)$.

By fundamental calculation, we get
\begin{eqnarray}
&&2\text{Re}\left(\sigma M_{11}(u,v),\sigma M_{12}(u,v)\right)=I_1+I_2+I_3,\\
&&2\text{Re}\left( M_{21}(u,v),M_{22}(u,v)\right)=J_1+J_2+J_3,
\end{eqnarray}
where
\begin{equation}\nonumber
\begin{array}{ll}
\ds I_1=2\textrm{Re}\int_{Q_T}(2s\sigma\nabla\varphi\cdot\nabla u+s\sigma\Delta\varphi u)\big(-i(\sigma b_{11} \bar{u}_t+\sigma b_{12} \bar{v}_t)+\sigma\Delta \bar{u}+s^2\sigma|\nabla \varphi|^2\bar{u}\big)dxdt,  \\
\ns\ds I_2=2\textrm{Re}\int_{Q_T}i(s\sigma b_{11}\varphi_tu+s\sigma b_{12}\varphi_tv)(-i(\sigma b_{11}\bar{u}_t+\sigma b_{12} \bar{v}_t)+\sigma\Delta \bar{u})dxdt,   \\
\ns\ds I_3=2\textrm{Re}\int_{Q_T}i(s\sigma b_{11}\varphi_tu+s\sigma b_{12}\varphi_tv)(\sigma s^2|\nabla \varphi|^2\bar{u} )dxdt.
 \end{array}
\end{equation}
and
\begin{equation}\nonumber
\begin{array}{ll}
\ds J_1=2\textrm{Re}\int_{Q_T}(2s\nabla\varphi\cdot\nabla v+s\Delta\varphi v)\big(-i( b_{21} \bar{u}_t+ b_{22} \bar{v}_t)+\Delta \bar{v}+s^2|\nabla \varphi|^2\bar{v}\big)dxdt, \qquad\qquad \\
\ns\ds J_2=2\textrm{Re}\int_{Q_T}i(sb_{21}\varphi_tu+sb_{22}\varphi_tv)( -i(b_{21} \bar{u}_t+ b_{22} \bar{v}_t)+\Delta \bar{v})dxdt,  \\
\ns\ds J_3=2\textrm{Re}\int_{Q_T}i(s b_{21}\varphi_tu +sb_{22}\varphi_tv)(s^2|\nabla \varphi|^2\bar{v})dxdt.
 \end{array}
\end{equation}
By integration by parts, there holds
\begin{align} \nonumber
I_1  & =-2\textrm{Re}\int_{Q_T}i(2s\sigma\nabla\varphi\cdot\nabla u +s\sigma\Delta\varphi u )\sigma b_{11}\bar{u}_tdxdt \\\nonumber
&\quad- 2\textrm{Re}\int_{Q_T}i(2s\sigma\nabla\varphi\cdot\nabla u +s\sigma\Delta\varphi u )\sigma b_{12}\bar{v}_tdxdt \\  \nonumber
& \quad + 2\textrm{Re}\int_{Q_T}(2s\sigma\nabla\varphi\cdot\nabla u +s\sigma\Delta\varphi u )(\sigma\Delta \bar{u} +s^2\sigma|\nabla \varphi|^2\bar{u} )dxdt\\
&=: 2\textrm{Re}(I_1^1+I_1^2+I_1^3).
\end{align}
Since $2\textrm{Re}(z)=z+\bar{z}$, we have
\begin{equation}\label{a11-1}
\begin{aligned}
 2\textrm{Re}(I_1^1)&= -\int_{Q_T}i(2s\sigma\nabla\varphi\cdot\nabla u+s\sigma\Delta\varphi u)\sigma b_{11}\bar{u}_tdxdt\\
&\q+\int_{Q_T}i(2s\sigma\nabla\varphi\cdot\nabla \bar{u}+s\sigma\Delta\varphi\bar{u})\sigma b_{11}u_tdxdt.
\end{aligned}
\end{equation}
Noting that $u(x,t)=v(x,t)=0$ on $\Gamma_T$, and $\lim\limits_{t\rightarrow 0}e^{-s\varphi(x,t)}=\lim\limits_{t\rightarrow T}e^{-s\varphi(x,t)}=0$, we obtain
\begin{equation}\label{c1}
\begin{aligned}
 & \quad-\int_{Q_T}2is\sigma^2b_{11}\nabla\varphi\cdot\nabla u \bar{u}_tdxdt+\int_{Q_T}2is\sigma^2b_{11}\nabla\varphi\cdot\nabla \bar{u}u_tdxdt   \\
&= \int_{Q_T}2is\sigma^2b_{11}\nabla\varphi_t\cdot\nabla u \bar{u}dxdt- \int_{Q_T}2is\nabla(\sigma^2b_{11})\cdot\nabla\varphi \bar{u}u_tdxdt \\
&\quad-\int_{Q_T}2is\sigma^2b_{11}\Delta\varphi\bar{u}u_tdxdt,
\end{aligned}
\end{equation}
and
\begin{equation}\label{c2}
\begin{aligned}
\int_{Q_T}is\sigma^2b_{11}\Delta\varphi (-u\bar{u}_t+\bar{u}u_t)dxdt=\int_{Q_T}is\sigma^2b_{11}(\Delta\varphi_t u+2\Delta\varphi u_t)\bar{u}dxdt.
\end{aligned}
\end{equation}
Substituting \eqref{c1} and \eqref{c2} into \eqref{a11-1} yields
\begin{equation}\label{a11-2}
\begin{aligned}
2\textrm{Re}(I_1^1)= -\textrm{Im}\int_{Q_T}2s\sigma^2b_{11}\nabla\varphi_t\cdot\nabla u \bar{u}dxdt+\textrm{Im}\int_{Q_T}2s\nabla(\sigma^2b_{11})\cdot\nabla\varphi \bar{u}u_tdxdt.
\end{aligned}
\end{equation}
Similarly, 
we have
\begin{equation}\label{c5}
\begin{aligned}
 2\textrm{Re}(I_1^2+J_1^1)&=-\textrm{Im}\int_{Q_T}4s\sigma^2b_{12}\nabla\varphi_t\cdot\nabla u\bar{v}dxdt-\textrm{Im}\int_{Q_T}2s\sigma^2b_{12}\Delta\varphi_t u\bar{v}dxdt\\
&\q-\textrm{Im}\int_{Q_T}4s\nabla b_{21}\cdot\nabla\varphi v\bar{u}_tdxdt.
\end{aligned}
\end{equation}
Moreover, we get
\begin{align}\label{c7}\nonumber
\textrm{Re}\int_{Q_T}2s\sigma^2\nabla\varphi\cdot\nabla u\Delta \bar{u}dxdt& \! = \! \int_{\Gamma_T}s\sigma^2\frac{\partial\varphi}{\partial \nu}\left|\frac{\partial u}{\partial \nu}\right|^2dSdt \! - \! \textrm{Re}   \int_{Q_T}2s(\nabla\sigma^2\cdot\nabla\bar{u})(\nabla\varphi\cdot\nabla u)dxdt \\\nonumber
&\q+ \int_{Q_T}s(\nabla\sigma^2\cdot\nabla \varphi)|\nabla u|^2dxdt+\int_{Q_T}s\sigma^2\Delta\varphi|\nabla u|^2dxdt\\
&\q -\textrm{Re}  \int_{Q_T}2s\sigma^2\sum_{i,j=1}^N\partial_j\partial_i\varphi\partial_i u\partial_j \bar{u}dxdt,
\end{align}
and
\begin{align}\nonumber
\textrm{Re}\int_{Q_T}s\sigma^2\Delta\varphi u\Delta\bar{u}dxdt  & = -\textrm{Re}\int_{Q_T}s\nabla(\sigma^2\Delta\varphi)\cdot\nabla\bar{u}udxdt-\int_{Q_T}s\sigma^2\Delta\varphi |\nabla u|^2dxdt\qquad\\
&= \frac{1}{2}\int_{Q_T}s\Delta(\sigma^2\Delta\varphi)|u|^2dxdt-\int_{Q_T}s\sigma^2\Delta\varphi |\nabla u|^2dxdt.
\end{align}
On the other hand,
\begin{equation}\label{c9}
\begin{aligned}
\textrm{Re}\int_{Q_T}2s^3\sigma^2|\nabla\varphi|^2\nabla\varphi\cdot\nabla u\bar{u}dxdt & = -\int_{Q_T}s^3\nabla(\sigma^2|\nabla\varphi|^2)\cdot\nabla\varphi|u|^2dxdt\\
&\q-\int_{Q_T}s^3\sigma^2|\nabla\varphi|^2\Delta\varphi|u|^2dxdt.
\end{aligned}
\end{equation}
This implies that
\begin{equation}\label{a11-3}
\begin{aligned}
2\textrm{Re}(I_1^3)&=  \int_{\Gamma_T}2s\sigma^2\frac{\partial\varphi}{\partial \nu}\left|\frac{\partial u}{\partial \nu}\right|^2dSdt - \textrm{Re}\int_{Q_T}4s(\nabla\sigma^2\cdot\nabla\bar{u})(\nabla\varphi\cdot\nabla u)dxdt \\
&\q+2\int_{Q_T}s(\nabla\sigma^2\cdot\nabla \varphi)|\nabla u|^2dxdt- \textrm{Re}\int_{Q_T}4s\sigma^2\sum_{i,j=1}^N\partial_j\partial_i\varphi\partial_i u\partial_j \bar{u}dxdt  \\
& \quad +\int_{Q_T}s\Delta(\sigma^2\Delta\varphi)|u|^2dxdt-2\int_{Q_T}s^3\nabla(\sigma^2|\nabla\varphi|^2)
\cdot\nabla\varphi|u|^2dxdt.
\end{aligned}
\end{equation}
$J_1^2$ and $J_1^3$ can be computed by similar process. Hence,
\begin{align}\nonumber
 I_1+J_1&=  \int_{\Gamma_T}2s\sigma^2\frac{\partial\varphi}{\partial \nu}\left|\frac{\partial u}{\partial \nu}\right|^2dSdt+ \int_{\Gamma_T}2s\frac{\partial\varphi}{\partial \nu}|\frac{\partial v}{\partial \nu}|^2dSdt\\\nonumber
 &\q -\textrm{Re} \int_{Q_T}4s\sigma^2\sum_{i,j=1}^N\partial_j\partial_i\varphi\partial_i u\partial_j \bar{u}dxdt   - \textrm{Re}\int_{Q_T}4s\sum_{i,j=1}^N\partial_j\partial_i\varphi\partial_i v\partial_j \bar{v}dxdt \\\nonumber
 &\q-\textrm{Re}\int_{Q_T}4s(\nabla\sigma^2\cdot\nabla\bar{u})(\nabla\varphi\cdot\nabla u)dxdt+ \int_{Q_T}2s(\nabla\sigma^2\cdot\nabla \varphi)|\nabla u|^2dxdt\\\nonumber
&\q+ \int_{Q_T}s\Delta(\sigma^2\Delta\varphi)|u|^2dxdt \! + \! \int_{Q_T}s\Delta^2\varphi|v|^2dxdt \! - \! \int_{Q_T}2s^3\nabla(|\nabla\varphi|^2)\cdot\nabla\varphi|v|^2dxdt\\\nonumber
&\q- \int_{Q_T}2s^3\nabla(\sigma^2|\nabla\varphi|^2)\cdot\nabla\varphi|u|^2dxdt - \textrm{Im}\int_{Q_T}2s\sigma^2b_{11}\nabla\varphi_t\cdot\nabla u \bar{u}dxdt\\\nonumber
&\q  + \textrm{Im}\int_{Q_T}2s\nabla(\sigma^2b_{11})\cdot\nabla\varphi \bar{u}u_tdxdt - \textrm{Im}\int_{Q_T}2sb_{22}\nabla\varphi_t\cdot\nabla v \bar{v}dxdt\\\nonumber
&\q +\textrm{Im}\int_{Q_T}2s\nabla b_{22}\cdot\nabla\varphi \bar{v}v_tdxdt- \textrm{Im}\int_{Q_T}4s\sigma^2b_{12}\nabla\varphi_t\cdot\nabla u\bar{v}dxdt\\\label{a11-7}
&\q  - \textrm{Im}\int_{Q_T}2s\sigma^2b_{12}\Delta\varphi_t u\bar{v}dxdt - \textrm{Im}\int_{Q_T}4s\nabla b_{21}\cdot\nabla\varphi v\bar{u}_tdxdt.
\end{align}
Next we calculate $I_2+J_2$.
\begin{equation}\label{a11-6}
\begin{aligned}
I_2 & = 2\textrm{Re}\int_{Q_T}i(s\sigma b_{11}\varphi_tu+s\sigma b_{12}\varphi_tv)(-i(\sigma b_{11}\bar{u}_t+\sigma b_{12}\bar{v}_t)+\sigma\Delta \bar{u})dxdt  \\
&=2\textrm{Re}\int_{Q_T}s\sigma^2\varphi_t(b_{11}^2u\bar{u}_t+b_{11}b_{12}u\bar{v}_t+b_{11}b_{12}v\bar{u}_t+b_{12}^2v\bar{v}_t)dxdt\\
&\quad -2\textrm{Im}\int_{Q_T}s\sigma^2b_{11}\varphi_tu\Delta\bar{u}dxdt -2\textrm{Im}\int_{Q_T}s\sigma^2b_{12}\varphi_tv\Delta\bar{u}dxdt,
\end{aligned}
\end{equation}
\begin{equation}
\begin{aligned}
J_2 & =2\textrm{Re}\int_{Q_T}s\varphi_t(b_{21}^2u\bar{u}_t+b_{21}b_{22}u\bar{v}_t+b_{21}b_{22}v\bar{u}_t+b_{22}^2v\bar{v}_t)dxdt\\
&\quad -2\textrm{Im}\int_{Q_T}sb_{21}\varphi_tu\Delta\bar{v}dxdt-2\textrm{Im}\int_{Q_T}sb_{22}\varphi_tv\Delta\bar{v}dxdt.
\end{aligned}
\end{equation}
Integrating by parts with respect to $t$, we get
\begin{equation}\label{a11-4}
\begin{aligned}
2\textrm{Re}\int_{Q_T}s\sigma^2b_{11}^2\varphi_tu\bar{u}_tdxdt&= \int_{Q_T}s\sigma^2b_{11}^2\varphi_t(u\bar{u}_t + \bar{u}u_t)dxdt \\ &=-\int_{Q_T}s\sigma^2b_{11}^2\varphi_{tt}|u|^2dxdt,
\end{aligned}
\end{equation}
\begin{equation}
\begin{aligned}
2\textrm{Re}\int_{Q_T}s\sigma^2b_{12}^2\varphi_tv\bar{v}_tdxdt& = \int_{Q_T}s\sigma^2b_{12}^2\varphi_t(v\bar{v}_t + \bar{v}v_t)dxdt\\
&= - \int_{Q_T}s\sigma^2b_{12}^2\varphi_{tt}|v|^2dxdt,
\end{aligned}
\end{equation}
\begin{equation}
\begin{aligned}
  2\textrm{Re}\int_{Q_T}s\sigma^2b_{11}b_{12}\varphi_t(u\bar{v}_t+\bar{u}_tv)dxdt=
-\textrm{Re}\int_{Q_T}2s\sigma^2b_{11}b_{12}\varphi_{tt}u\bar{v}dxdt.
\end{aligned}
\end{equation}
On the other hand,
\begin{equation}
\begin{aligned}
-2\textrm{Im}\int_{Q_T}s\sigma^2b_{11}\varphi_tu\Delta\bar{u}dxdt&=\int_{Q_T}is\sigma^2b_{11}\varphi_t(u\Delta\bar{u}-\bar{u}\Delta u)dxdt\\
&=\textrm{Im}\int_{Q_T}2s\nabla(\sigma^2b_{11}\varphi_t)\cdot\nabla \bar{u}udxdt,
\end{aligned}
\end{equation}
and
\begin{align}\nonumber
\int_{Q_T}s\sigma^2b_{12}\varphi_tv\Delta\bar{u}dxdt  & = -\int_{Q_T}s\nabla(\sigma^2b_{12}\varphi_t)\cdot\nabla\bar{u}vdxdt
-\int_{Q_T}s\sigma^2b_{12}\varphi_t\nabla v\cdot\nabla\bar{u}dxdt\qquad\qquad  \\\nonumber
& = \int_{Q_T}s\Delta(\sigma^2b_{12}\varphi_t)\bar{u}vdxdt+2\int_{Q_T}s\nabla(\sigma^2b_{12}\varphi_t)\cdot\nabla v\bar{u}dxdt \\\label{a11-5}
&\quad +\int_{Q_T}s\sigma^2b_{12}\varphi_t\Delta v\bar{u}dxdt.
\end{align}
Substituting \eqref{a11-4}--\eqref{a11-5} into \eqref{a11-6}, and with similar analysis to $J_2$, we have
\begin{align}\nonumber
I_2+J_2 & = -\int_{Q_T}s\sigma^2b_{11}^2\varphi_{tt}|u|^2dxdt-\int_{Q_T}s\sigma^2b_{12}^2\varphi_{tt}|v|^2dxdt  \\\nonumber
& \quad -\int_{Q_T}sb_{21}^2\varphi_{tt}|u|^2dxdt-\int_{Q_T}sb_{22}^2\varphi_{tt}|v|^2dxdt\\\nonumber
&\q-\textrm{Re}\int_{Q_T}2s\sigma^2b_{11}b_{12}\varphi_{tt}u\bar{v}dxdt-\textrm{Re}\int_{Q_T}2sb_{21}b_{22}\varphi_{tt}\bar{u}vdxdt \\\nonumber
& \quad +\textrm{Im}\int_{Q_T}2s\nabla(\sigma^2b_{11}\varphi_t)\cdot\nabla\bar{u}u dxdt+\textrm{Im}\int_{Q_T}2s\nabla(b_{22}\varphi_t)\cdot\nabla\bar{v}vdxdt \\\label{a11-8}
&\q-\textrm{Im}\int_{Q_T}2s\Delta(\sigma^2b_{12}\varphi_t)\bar{u}vdxdt-\textrm{Im}\int_{Q_T}4s\nabla(\sigma^2b_{12}\varphi_t)\cdot\nabla v\bar{u}dxdt.
\end{align}
Noting that
\begin{equation}
\begin{aligned}
I_3&=2\textrm{Re}\int_{Q_T}i(s\sigma b_{11}\varphi_tu+s\sigma b_{12}\varphi_tv)(\sigma s^2|\nabla \varphi|^2\bar{u})dxdt\\
&=-\textrm{Im}\int_{Q_T}2s^3\sigma^2b_{12}|\nabla \varphi|^2\varphi_tv\bar{u}dxdt,
\end{aligned}
\end{equation}
\begin{equation}
\begin{aligned}
J_3 &= 2\textrm{Re}\int_{Q_T}i(s b_{21}\varphi_tu + s b_{22}\varphi_tv)( s^2|\nabla \varphi|^2\bar{v})dxdt\qq \\
&= -\textrm{Im}\int_{Q_T}2s^3b_{21}|\nabla \varphi|^2\varphi_tu\bar{v}dxdt.
\end{aligned}
\end{equation}
Thus
\begin{equation}\label{a11-9}
 I_3+J_3=0.
 \end{equation}
Substituting \eqref{a11-7}, \eqref{a11-8} and \eqref{a11-9} into \eqref{a11-10}, and noting that
\begin{enumerate}[1)]
      \item  $\partial_i\varphi=-\lambda \theta\partial_i\psi,\q \partial_j\partial_i\varphi=-\theta(\lambda^2\partial_i\psi\partial_j\psi+\lambda\partial_j\partial_i\psi),$
      \item  $\frac{\partial\varphi}{\partial \nu}= -\lambda\theta\frac{\partial\psi}{\partial \nu},$
      \item  $|\varphi_t|\leq C_1\theta^2$, $|\varphi_{tt}|\leq C_2\theta^3,\q \exists C_1,C_2>0,$
      \item $-\partial_j\partial_i\varphi\partial_i u\partial_j \bar{u}=\lambda\theta(\lambda|\nabla\psi\cdot\nabla u|^2+\partial_j\partial_i\psi\partial_i u\partial_j \bar{u}),$
      \item for $\lambda$ sufficiently large and $C>0$, there holds
          $-\nabla(|\nabla\varphi|^2)\cdot\nabla\varphi\geq C\lambda^4\theta^3|\nabla\psi|^4,$
 \end{enumerate}
 then, for $s$ and $\lambda$ sufficiently large, there holds
 \begin{align}\label{d1}\nonumber
\|\sigma G_1\|+\| G_2\| &\geq \|\sigma M_{11}(u ,v)\|+\|\sigma M_{12}(u ,v)\|+\| M_{21}(u ,v)\|+\| M_{22}(u ,v)\|\\\nonumber
&\q - \int_{\Gamma_T}2s\lambda\theta\frac{\partial\psi}{\partial \nu}\left(\sigma^2\left|\frac{\partial u}{\partial \nu}\right|^2+\left|\frac{\partial v}{\partial \nu}\right|^2\right)dSdt-\textrm{Im}\int_{Q_T}4s\nabla b_{21}\cdot\nabla\varphi v\bar{u}_tdxdt\\\nonumber
&\q   + \textrm{Im}\int_{Q_T}2s\nabla(\sigma^2b_{11})\cdot\nabla\varphi \bar{u}u_tdxdt+\textrm{Im}\int_{Q_T}2s\nabla b_{22}\cdot\nabla\varphi \bar{v}v_tdxdt   \\\nonumber
 &\q+ \textrm{Re}\int_{Q_T}4\sigma^2s\lambda\theta\big(\lambda|\nabla\psi\cdot\nabla u|^2+\sum_{i,j=1}^N\partial_j\partial_i\psi\partial_i u\partial_j \bar{u}\big)dxdt\\\nonumber
 &\q +\textrm{Re}\int_{Q_T}4s\lambda\theta\big(\lambda|\nabla\psi\cdot\nabla v|^2+\sum_{i,j=1}^N\partial_j\partial_i\psi\partial_i v\partial_j \bar{v}\big)dxdt\\\nonumber
 &\q +  \textrm{Re}\int_{Q_T}4s\lambda\theta(\nabla\sigma^2\cdot\nabla\bar{u})(\nabla\psi\cdot\nabla u)dxdt- \int_{Q_T}2s\lambda\theta(\nabla\sigma^2\cdot\nabla \psi)|\nabla u|^2dxdt\\\nonumber
 &\q - \textrm{Re}\int_{Q_T}Cs\theta^3\sigma^2b_{11}b_{12}u\bar{v}dxdt-  \textrm{Re}\int_{Q_T}Cs\theta^3b_{21}b_{22}\bar{u}vdxdt\\\nonumber
 &\q - \textrm{Im}\int_{Q_T}Cs\theta^2\Delta b_{21}\bar{u}vdxdt - \textrm{Im}\int_{Q_T}Cs\lambda\theta^2\nabla b_{21}\cdot\nabla\psi\bar{u}vdxdt  \\\nonumber
 &\q -\textrm{Im}\int_{Q_T}Cs\lambda\theta^2\sigma^2b_{11}\nabla\psi\cdot\nabla u \bar{u}dxdt-  \textrm{Im}\int_{Q_T}Cs\lambda\theta^2b_{22}\nabla\psi\cdot\nabla v \bar{v}dxdt\\\nonumber
 &\q - \textrm{Im}\int_{Q_T}Cs\theta^2\nabla(\sigma^2b_{11})\cdot\nabla u\bar{u}dxdt- \textrm{Im}\int_{Q_T}Cs\theta^2\nabla b_{22}\cdot\nabla v\bar{v}dxdt\\\nonumber
 &\q - \textrm{Im}\int_{Q_T}Cs\lambda\theta^2b_{21}\nabla\psi\cdot\nabla u\bar{v}dxdt-  \textrm{Im}\int_{Q_T}Cs\lambda\theta^2b_{21}\nabla\psi\cdot\nabla v\bar{u}dxdt \\
 &\q   +\int_{Q_T}Cs^3\lambda^4\theta^3\sigma^2|\nabla\psi|^4|u|^2dxdt+ \int_{Q_T}Cs^3\lambda^4\theta^3|\nabla\psi|^4|v|^2dxdt.
\end{align}
By $\textrm{H\"{o}lder's}$ inequality, and for $s$ and $\lambda$ sufficiently large, we have
\begin{align}\label{d3}\nonumber
\|\sigma G_1\|+\| G_2\|&\geq \|\sigma M_{11}(u ,v)\|+\|\sigma M_{12}(u ,v)\|+\| M_{21}(u ,v)\|+\| M_{22}(u ,v)\|\\\nonumber
&\q  - \int_{\Gamma_T}2s\lambda\theta\frac{\partial\psi}{\partial \nu}\left(\sigma^2\left|\frac{\partial u}{\partial \nu}\right|^2+\left|\frac{\partial v}{\partial \nu}\right|^2\right)dSdt-\textrm{Im}\int_{Q_T}4s\nabla b_{21}\cdot\nabla\varphi v\bar{u}_tdxdt\\\nonumber
&\q   + \textrm{Im}\int_{Q_T}2s\nabla(\sigma^2b_{11})\cdot\nabla\varphi \bar{u}u_tdxdt+\textrm{Im}\int_{Q_T}2s\nabla b_{22}\cdot\nabla\varphi \bar{v}v_tdxdt   \\\nonumber
&\q  +C\int_{Q_T}s^3\lambda^4\theta^3\big(|u|^2+|v|^2\big)dxdt+C\int_{Q_T}s\lambda^2\theta\big(|\nabla u|^2+|\nabla v|^2\big)dxdt \\
&\q  -C\int_{\omega_T}s^3\lambda^4\theta^3\big(|u|^2+|v|^2\big)dxdt+C\int_{\omega_T}s\lambda^2\theta\big(|\nabla u|^2+|\nabla v|^2\big)dxdt.
\end{align}
Further, we need to deal with $\int_{Q_T}2s\nabla(\sigma^2b_{11})\cdot\nabla\varphi \bar{u}u_tdxdt$, $\int_{Q_T}2s\nabla b_{22}\cdot\nabla\varphi \bar{v}v_tdxdt$ and $-\int_{Q_T}4s\nabla b_{21}\cdot\nabla\varphi v\bar{u}_tdxdt$, such that $u_t$ and $v_t$ can be simultaneously absorbed by the other terms. By substituting \eqref{b6} into \eqref{b1} and multiplying the first and second equations in \eqref{b1} by $e^{-s\varphi}$, we obtain:
\begin{equation}
\begin{aligned}
 iu_t&=e^{-s\varphi}f_1-is\varphi_tu-s^2a_{11}|\nabla\varphi|^2u-2sa_{11}\nabla\varphi\cdot\nabla u-sa_{11}\Delta\varphi u-a_{11}\Delta u\\
&\q-s^2a_{12}|\nabla\varphi|^2v-2sa_{12}\nabla\varphi\cdot\nabla v-sa_{12} \Delta\varphi v-a_{12}\Delta v,
\end{aligned}
\end{equation}
\begin{equation}
\begin{aligned}
 iv_t&=e^{-s\varphi}f_2-is\varphi_tv-s^2a_{21}|\nabla\varphi|^2u-2sa_{21}\nabla\varphi\cdot\nabla u-sa_{21}\Delta\varphi u-a_{21}\Delta u\\
&\q-s^2a_{22}|\nabla\varphi|^2v-2sa_{22}\nabla\varphi\cdot\nabla v-sa_{22} \Delta\varphi v-a_{22}\Delta v.
\end{aligned}
\end{equation}
Then
\begin{align}\label{e3}\nonumber
&\q\textrm{Im}\int_{Q_T}2s\nabla(\sigma^2b_{11})\cdot\nabla\varphi \bar{u}u_tdxdt\\\nonumber
&=-\textrm{Re}\int_{Q_T}2s\nabla(\sigma^2b_{11})\cdot\nabla\varphi \bar{u}(iu_t)dxdt\\\nonumber
&=\textrm{Re}\left\{ -\int_{Q_T}2s\nabla(\sigma^2b_{11})\cdot\nabla\varphi \bar{u}f_1e^{-s\varphi}dxdt+\int_{Q_T}2s^3a_{11}|\nabla\varphi|^2\nabla(\sigma^2b_{11})\cdot\nabla\varphi|u|^2dxdt\right.\\\nonumber
&\qq\q \!-\! \int_{Q_T}2s^2\nabla(a_{11}\nabla(\sigma^2b_{11})\cdot\nabla\varphi)\cdot
\nabla\varphi|u|^2dxdt\!-\!\int_{Q_T}2s\nabla(a_{11}\nabla(\sigma^2b_{11})\cdot\nabla\varphi)\cdot\nabla u\bar{u}dxdt \\\nonumber
&\qq\q-\int_{Q_T}2sa_{11}\nabla(\sigma^2b_{11})\cdot\nabla\varphi|\nabla u|^2dxdt+\int_{Q_T}2s^3a_{12}|\nabla\varphi|^2\nabla(\sigma^2b_{11})\cdot\nabla\varphi\bar{u}vdxdt\\\nonumber
&\qq\q \! + \! \int_{Q_T}4s^2a_{12}\nabla(\sigma^2b_{11})\cdot\nabla\varphi\nabla\varphi\cdot\nabla v\bar{u}dxdt+\int_{Q_T}2s^2a_{12}\nabla(\sigma^2b_{11})\cdot\nabla\varphi\Delta\varphi\bar{u}vdxdt\\
&\left.\qq\q-\int_{Q_T}2s\nabla(a_{12}\nabla(\sigma^2b_{11})\cdot\nabla\varphi)\cdot\nabla v\bar{u}dxdt-\int_{Q_T}2sa_{12}\nabla(\sigma^2b_{11})\cdot\nabla\varphi\nabla v\cdot\nabla\bar{u}dxdt\right\}.
\end{align}
 Similarly, we can handle $\int_{Q_T}2s\nabla b_{22}\cdot\nabla\varphi \bar{v}v_tdxdt$ and $-\int_{Q_T}4s\nabla b_{21}\cdot\nabla\varphi v\bar{u}_tdxdt$, and using Young's inequality we have
\begin{align}\label{e4}\nonumber
&\q\textrm{Im}\left\{-\int_{Q_T}4s\nabla b_{21}\cdot\nabla\varphi v\bar{u}_tdxdt \! + \! \int_{Q_T}2s\nabla(\sigma^2b_{11})\cdot\nabla\varphi \bar{u}u_tdxdt \! + \! \int_{Q_T}2s\nabla b_{22}\cdot\nabla\varphi \bar{v}v_tdxdt \right\}  \\\nonumber
&\geq -\varepsilon\int_{Q_T}\left(|f_1|^2+|f_2|^2\right)e^{-2s\varphi}dxdt-C_\varepsilon\int_{Q_T}s^2\lambda^2\theta^3\left(|u|^2+|v|^2\right)dxdt\\
&\q-C\int_{Q_T}\left(s^3\lambda^3\theta^3+s^2\lambda^3\theta^2+s\lambda^3\theta+s\lambda\theta\right)\left(|u|^2+|v|^2\right)dxdt\\\nonumber
&\q-C\int_{Q_T}s\lambda\theta\left(|\nabla u|^2+|\nabla v|^2\right)dxdt,
\end{align}
where $\varepsilon>0$ is a sufficiently small constant. Thus \eqref{e5} can be obtained by using \eqref{b4}, \eqref{b6}, \eqref{b11}, \eqref{d3} and \eqref{e4}.
\end{proof}
The following Carleman estimate is a direct consequence of the proposition \ref{A}. Consider the system \eqref{a1}, we have the following result.
\bt\label{B}
Let $\psi$, $\varphi$ and $\theta$ be as in proposition \ref{A}, let $f_1$, $f_2\in L^2(\overline{Q_T})$, and $(y_1,y_2)\in [H^1(0,T;L^2(\Omega))\cap L^2(0,T;H^2(\Omega))]^2$ is the solution of \eqref{a1} with $g_1(x,t)=g_2(x,t)=0$, $a_{ij}\in C^2(\overline{\Omega})$ and satisfy \eqref{ap}. Then there exist constants $s_0\geq1$, $\lambda_0\geq1$ and $C>0$ such that for all $s>s_0$, $\lambda>\lambda_0$, the following inequality holds:
\begin{align}\label{e6}\nonumber
&\q  \int_{Q_T}\left(|\tilde{M}_{11}(y_1 ,y_2)|^2+| \tilde{M}_{12}(y_1 ,y_2)|^2+| \tilde{M}_{21}(y_1 ,y_2)|^2+| \tilde{M}_{22}(y_1 ,y_2)|^2\right)e^{-2s\varphi}dxdt\\\nonumber
&\q  +\int_{Q_T}\left[s^3\lambda^4\theta^3\big(|y_1|^2+|y_2|^2\big)+s\lambda^2\theta\big(|\nabla y_1|^2+|\nabla y_2|^2\big)\right]e^{-2s\varphi}dxdt\\\nonumber
&\q  +\int_{\Gamma_T}s\lambda\theta\left(\left|\frac{\partial y_1}{\partial \nu}\right|^2+\left|\frac{\partial y_2}{\partial \nu}\right|^2\right)e^{-2s\varphi}dSdt\\\nonumber
&\leq C\int_{Q_T}\left(|f_1|^2+|f_2|^2\right)e^{-2s\varphi}dxdt + C\int_{\omega_T}s^3\lambda^4\theta^3\big(|y_1|^2+|y_2|^2\big)e^{-2s\varphi}dxdt\\
&\q+ C\int_{\omega_T}s\lambda^2\theta\big(|\nabla y_1|^2+|\nabla y_2|^2\big)e^{-2s\varphi}dxdt,
\end{align}
where $\tilde{M}_{11}$, $\tilde{M}_{12}$, $\tilde{M}_{21}$ and $\tilde{M}_{22}$ are defined in \eqref{h1} and \eqref{h2}.
\et
\subsection{IP (I)}

For the inverse problem, we consider the system \eqref{a1},
The following assumptions are needed in the proof of stability estimate of IP (I) and IP (II).
\begin{itemize}\label{a11-15}
      \item[\textit{A1.}] $a_{ij}\in C^2(\overline{\Omega})$ such that
          $$a_{12}(x)a_{21}(x)>0,\ \textrm{det}(a_{ij})\neq0,\ a_{22}(x)\cdot\textrm{det}(a_{ij})>0,\qq x\in\overline{\Omega}.$$
      \item[\textit{A2.}] there exists a $\sigma_0>0$ such that $\sigma(x)>\sigma_0$ and
             $$\sigma^2(x)a_{12}(x)=a_{21}(x).$$
      \item[\textit{A3.}]  $a,\tilde{a},b,c,d\in C(\bar{\Omega});$
      \item[\textit{A4.}]  $y_{10}(x)\in \mathbb{R}$ or $iy_{10}(x)\in \mathbb{R}$ a.e. in $\Omega.$
      \item[\textit{A5.}]  there exist a positive constant $r$ such that
      $$|y_{10}|\geq r >0\q \textrm{a.e. in}\ \Omega.$$
 \end{itemize}
\begin{theorem}\label{C}
Let $(y_1,y_2)\in [H^2(0,T;L^2(\Omega))\cap H^1(0,T;H^2(\Omega))]^2$ is the solution of \eqref{a1}, and the assumptions $A1-A5$ hold. Then there exists a constant $C>0$ such that
\begin{align}\label{f2}\nonumber
\|\tilde{a}-a\|^2_{L^2(\Omega)}\leq C\left(\|\tilde{y}_1(\tilde{a})-y_1(a)\|^2_{H^1(0,T;H^1(\omega))}+\|\tilde{y}_2(\tilde{a})-y_2(a)\|^2_{H^1(0,T;H^1(\omega))}\right),
\end{align}
where $(y_1,y_2)$, $(\tilde{y}_1,\tilde{y}_2)$ are solutions of \eqref{a1} corresponding to $a$ and $\tilde{a}$, respectively.
\end{theorem}
\begin{proof}
For the sake of simplicity, we set $z_1=y_1(a)-\tilde{y}_1(\tilde{a})$, $z_2=y_2(a)-\tilde{y}_2(\tilde{a})$, thus $(z_1,z_2)$ satisfies
\begin{equation}\label{f2}
  \left\{
   \begin{array}{ll}
   \ds i\partial_t z_1+a_{11}\Delta z_1+a_{12}\Delta z_2+a(x)z_1+b(x)z_2=f(x)R(x,t),\quad &\textrm{in}\ \Omega\times(0,T),  \\
   \ns\ds i\partial_t z_2+a_{21}\Delta z_1+a_{22}\Delta z_2+c(x)z_1+d(x)z_2=0,\quad   &\textrm{in}\ \Omega\times(0,T),   \\
   \ns\ds z_1(x,t)=z_2(x,t)=0,\quad  &\textrm{on}\ \Gamma\times(0,T),  \\
   \ns\ds z_1(x,0)=z_2(x,0)=0,\quad&\textrm{in}\ \Omega,
   \end{array}
   \right.
  \end{equation}
with $f(x)=\tilde{a}-a$ and $R(x,t)=\tilde{y}_1(x,t)$.
Take the even-conjugate extension of $(z_1,z_2)$ to the interval $(-T,T)$, i.e., set
$$\left(z_1(x,t),z_2(x,t)\right)=\left(\overline{z_1(x,-t)},\overline{z_2(x,-t)}\right),\qq \textrm{for} \ t\in(-T,0).$$
If $R(x,0)\in \mathbb{R}$ for a.e. $x\in\Omega$, then we set
$$R(x,t)=\overline{R(x,-t)},\qq\textrm{for} \ t\in(-T,0).$$
If $iR(x,0)\in \mathbb{R}$ for a.e. $x\in\Omega$, then we set
$$R(x,t)=-\overline{R(x,-t)},\qq\textrm{for} \ t\in(-T,0).$$
Change $t$ into $t+T$, thus $(z_1,z_2)$ and $R(x,t)$ fulfill the system \eqref{f2} in $\Omega\times(0,2T)$. Let $(u_1,u_2)=(\partial_t z_1(x,2T-t),\partial_t z_2(x,2T-t))$, then by fundamental computation we have
\begin{equation}\label{f3}
  \left\{
   \begin{array}{ll}
   \ds i\partial_t u_1+a_{11}\Delta u_1+a_{12}\Delta u_2+a(x)u_1+b(x)u_2=f(x)R_t(x,t),\ &\textrm{in}\ \Omega\times(0,2T),  \\
   \ns\ds i\partial_t u_2+a_{21}\Delta u_1+a_{22}\Delta u_2+c(x)u_1+d(x)u_2=0,\   &\textrm{in}\ \Omega\times(0,2T),   \\
   \ns\ds u_1(x,t)=u_2(x,t)=0,\  &\textrm{on}\ \Gamma\times(0,2T),\\
   \ns\ds u_1(T)=-if(x)R(x,T),\ u_2(T)=0, \ &\textrm{in}\ \Omega.
   \end{array}
   \right.
  \end{equation}
Define
$$\theta(x,t)=\frac{e^{\lambda\psi(x)}}{t(2T-t)},\ \varphi(x,t)=\frac{e^{\lambda C_\psi}-e^{\lambda\psi(x)}}{t(2T-t)},\quad\quad\quad\forall (x,t)\in \Omega\times(0,2T).$$
Let $v_1=e^{-s\varphi}u_1$, $v_2=e^{-s\varphi}u_2$, recall that
\begin{equation}\label{f8}
\begin{cases}\ds
M_{12}(v_1,v_2)=i(b_{11}\partial_t v_1+b_{12}\partial_t v_2)+\Delta v_1+s^2|\nabla \varphi|^2v_1,\\
\ns\ds M_{22}(v_1,v_2)=i(b_{21}\partial_t v_1+b_{22}\partial_t v_2)+\Delta v_2+s^2|\nabla \varphi|^2v_2.
\end{cases}
\end{equation}
Then set
\begin{equation}
\begin{cases}\ds
\tilde{M}_{12}(u_1,u_2)=e^{s\varphi}M_{12}(v_1,v_2),\\
\ns\ds \tilde{M}_{22} (u_1,u_2)=e^{s\varphi}M_{22}(v_1,v_2),
\end{cases}
\end{equation}
and
\begin{equation}
\begin{cases}\ds
L_1  = \int_{0}^T\int_{\Omega}\sigma^2e^{-2s\varphi}\tilde{M}_{12}(u_1,u_2)\bar{u}_1dxdt,\\
\ns\ds L_2   = \int_{0}^T\int_{\Omega}e^{-2s\varphi}\tilde{M}_{22}(u_1,u_2)\bar{u}_2dxdt,
\end{cases}
\end{equation}
we can obtain
\begin{equation}
\begin{aligned}
\quad L_1  & = \int_{0}^T\int_{\Omega}\sigma^2e^{-2s\varphi}\tilde{M}_{12}(u_1,u_2)\bar{u}_1dxdt  \\
& = \int_{0}^T\int_{\Omega}\sigma^2M_{12}(v_1,v_2)\bar{v}_1dxdt\\
& = \int_{0}^T\int_{\Omega} \left(i(\sigma^2b_{11}\partial_tv_1\bar{v}_1+\sigma^2b_{12}\partial_tv_2\bar{v}_1)+\sigma^2(\Delta v_1+s^2|\nabla \varphi|^2v_1)\bar{v}_1\right)dxdt.
\end{aligned}
\end{equation}
\begin{equation}
\begin{aligned}
L_2  & = \int_{0}^T\int_{\Omega}e^{-2s\varphi}\tilde{M}_{22}(u_1,u_2)\bar{u}_2dxdt  \\
& = \int_{0}^T\int_{\Omega}M_{22}(v_1,v_2)\bar{v}_2dxdt\\
& = \int_{0}^T\int_{\Omega}\left(i(b_{21}\partial_tv_1\bar{v}_2+b_{22}\partial_tv_2\bar{v}_2)+(\Delta v_2+s^2|\nabla \varphi|^2v_2)\bar{v}_2\right)dxdt.\qq \
\end{aligned}
\end{equation}
By integrations by parts and $\textrm{H\"{o}lder's}$ inequality, we have
\begin{align}\nonumber\label{f4}
&\q\mathrm{Im}(L_1+L_2)\\\nonumber
 & = \frac{1}{2}\int_\Omega\sigma^2b_{11}|u_1(x,T)|^2e^{-2s\varphi(x,T)}dx - \mathrm{Im}\int_{0}^T\int_{\Omega}\nabla\sigma^2\cdot\nabla v_1 \bar{v}_1dxdt\\\nonumber
& = \frac{1}{2}\int_\Omega \sigma^2b_{11}e^{-2s\varphi(x,T)}|f(x)|^2|R(x,T)|^2dx - \mathrm{Im}\int_{0}^T\int_{\Omega}\nabla\sigma^2\cdot\nabla v_1 \bar{v}_1dxdt \\
& \geq C\int_\Omega e^{-2s\varphi(x,T)}|f(x)|^2dx-C\int_{0}^T\int_{\Omega}\left(s^{-\frac{1}{2}}|\nabla v_1|^2+s^{\frac{1}{2}}|v_1|^2\right)dxdt.
\end{align}
On the other hand,
\begin{align}\nonumber
|L_1|  & \leq \left(\int_{0}^T\int_{\Omega}\lambda^{-2}s^{-\frac{3}{2}}\sigma^2e^{-2s\varphi}|\tilde{M}_{12}|^2dxdt\right)^\frac{1}{2} \left(\int_{0}^T\int_{\Omega}\lambda^{2}s^{\frac{3}{2}}\sigma^2e^{-2s\varphi}|u_1 |^2dxdt\right)^\frac{1}{2}  \\\nonumber
& \leq \frac{1}{2}\left(\lambda^{-2}s^{-\frac{3}{2}}\int_{0}^T\int_{\Omega}\sigma^2e^{-2s\varphi}|\tilde{M}_{12}|^2dxdt
+\lambda^{2}s^{\frac{3}{2}}\int_{0}^T\int_{\Omega}\sigma^2e^{-2s\varphi}|u_1 |^2dxdt\right)  \\
&\leq C\lambda^{-2}s^{-\frac{3}{2}}\left(\int_{0}^T\int_{\Omega}e^{-2s\varphi}|\tilde{M}_{12}|^2dxdt+s^3\lambda^4\int_{0}^T\int_{\Omega}e^{-2s\varphi}|u_1 |^2dxdt\right).
\end{align}
By similar calculation to $L_2$ we have
\begin{align}\nonumber\label{f5}
|L_1+L_2|& \leq C\lambda^{-2}s^{-\frac{3}{2}}\left(\int_{0}^T\int_{\Omega}e^{-2s\varphi}(|\tilde{M}_{12}|^2+|\tilde{M}_{22}|^2)dxdt\right.\\
&\left. \qq\qq\qq+s^3\lambda^4\int_{0}^T\int_{\Omega}e^{-2s\varphi}(|u_1 |^2+|u_2|^2)dxdt\right).
\end{align}
From \eqref{f4} and \eqref{f5}, it follows that
\begin{align}\nonumber\label{f6}
&\q\int_\Omega e^{-2s\varphi(x,T)}|f(x)|^2dx\\\nonumber
& \leq C\lambda^{-2}s^{-\frac{3}{2}}\left(\int_{0}^T\int_{\Omega}e^{-2s\varphi}(|\tilde{M}_{12}|^2+|\tilde{M}_{22}|^2)dxdt+\int_{0}^T\int_{\Omega}s\lambda^2\theta e^{-2s\varphi}|\nabla u_1|^2dxdt\right.\\
&\left.\qq\qq\qq +s^3\lambda^4\int_{0}^T\int_{\Omega}\theta^3e^{-2s\varphi}(|u_1 |^2+|u_2|^2)dxdt\right).
\end{align}
Noting that $e^{-2s\varphi(x,t)}\leq e^{-2s\varphi(x,T)}$ for all $(x,t)\in\Omega\times(0,2T)$,
and apply Carleman inequality $\eqref{e6}$ with $2T$ instead of $t$ to \eqref{f6}, we obtain
\begin{align}\nonumber\label{f7}
&\q\int_\Omega e^{-2s\varphi(x,T)}|f(x)|^2dx\\\nonumber
& \leq C\lambda^{-2}s^{-\frac{3}{2}}\left(\int_{0}^{2T}\int_{\Omega}e^{-2s\varphi}|fR_t|^2dxdt+s^3\lambda^4\int_{0}^{2T}\int_{\omega}\theta^3e^{-2s\varphi}(|u_1 |^2+|u_2|^2)dxdt\right.\\\nonumber
&\left.\qq\qq\qq +\int_{0}^{2T}\int_{\omega}s\lambda^2\theta e^{-2s\varphi}(|\nabla u_1|^2+|\nabla u_2|^2)dxdt\right)\\\nonumber
&\leq C\lambda^{-2}s^{-\frac{3}{2}}\int_{\Omega}e^{-2s\varphi(x,T)}|f|^2dx+Cs^{-\frac{1}{2}}\int_{0}^{2T}\int_{\omega}(|\nabla u_1|^2+|\nabla u_2|^2)dxdt\\
&\q+Cs^{\frac{3}{2}}\lambda^2\int_{0}^{2T}\int_{\omega}(|u_1 |^2+|u_2|^2)dxdt.
\end{align}
Therefore, if $s$ and $\lambda$ sufficiently large, we deduce that
\begin{align}\nonumber
 \|f\|^2_{L^2(\Omega)}  & \leq C \left(\|u_1\|^2_{L^2(0,2T;H^1(\omega))}+\|u_2\|^2_{L^2(0,2T;H^1(\omega))}\right) \\
& \leq C\left(\|\partial_tz_1\|^2_{L^2(0,T;H^1(\omega))}+\|\partial_tz_2\|^2_{L^2(0,T;H^1(\omega))}\right).
\end{align}
That is
\begin{align}
\|\tilde{a}-a\|^2_{L^2(\Omega)}\leq C\left(\|y_1(\tilde{a})-y_1(a)\|^2_{H^1(0,T;H^1(\omega))}+\|y_2(\tilde{a})-y_2(a)\|^2_{H^1(0,T;H^1(\omega))}\right).
\end{align}
\end{proof}
\section{ IP (II): Boundary Observations}
 \subsection{Carleman estimate}
 Once more, let $\Omega\subset\mathbb{R}^N$ be a bounded domain with smooth boundary $\Gamma$, let $\Gamma^+$ be an open set of $\Gamma$. Suppose $\psi(x)\in C^4(\overline{\Omega})$ is a weight function which satisfied the following properties:
 \begin{itemize}
      \item  $\nabla\psi(x)\neq0,\quad\quad \textrm{in}\ \overline{\Omega}$.
      \item  $\frac{\partial\psi}{\partial \nu}\leq 0,\qquad \ \textrm{on}\ \Gamma\setminus\Gamma^+$.
      \item  $\frac{\partial\psi}{\partial \nu}> 0,\qquad \ \textrm{on}\ \Gamma^+$.
      \item  There exist a constant $\mu>0$, such that $\forall x\in \overline{\Omega}$ and $\forall \xi=(\xi_1,...,\xi_N)\in\mathbb{R}^N$,
      $$|\nabla\psi\cdot\xi|^2\geq\mu|\xi|^2.$$
      \item  $\psi(x)>\frac{3}{4}\|\psi\|_{L^\infty(\Omega)},  \quad\forall x\in\Omega$.
 \end{itemize}
Set $C_\psi=\frac{3}{2}\|\psi\|_{L^\infty(\Omega)}$, $\lambda>0$ and let
$$\theta(x,t)=\frac{e^{\lambda\psi(x)}}{t(T-t)},\ \varphi(x,t)=\frac{e^{\lambda C_\psi}-e^{\lambda\psi(x)}}{t(T-t)},\quad\quad\quad\forall (x,t)\in \Omega\times(0,T).$$
We prove here a Carleman estimate with a boundary observation acting on a subset $\Gamma^+$ of $\Gamma$. Consider the system \eqref{b1}, we have the following result.
\bp\label{D}
Suppose there exist $\psi$, $\varphi$ and $\theta$ satisfy the above conditions, let $f_1$, $f_2\in L^2(\overline{Q_T})$, and $(y_1,y_2)\in [H^1(0,T;L^2(\Omega))\cap L^2(0,T;H^2(\Omega))]^2$ is the solution of \eqref{b1}. Then there exist constants $s_0\geq1$, $\lambda_0\geq1$ and $C>0$ such that for all $s>s_0$, $\lambda>\lambda_0$, the next inequality holds:
\begin{align}\nonumber
&\q  \int_{Q_T}\left(|\tilde{M}_{11}(y_1 ,y_2)|^2+|\tilde{M}_{12}(y_1 ,y_2)|^2+| \tilde{M}_{21}(y_1 ,y_2)|^2+| \tilde{M}_{22}(y_1 ,y_2)|^2\right)e^{-2s\varphi}dxdt\\\nonumber
&\q  +\int_{Q_T}\left[s^3\lambda^4\theta^3\big(|y_1|^2+|y_2|^2\big)+s\lambda^2\theta\big(|\nabla y_1|^2+|\nabla y_2|^2\big)\right]e^{-2s\varphi}dxdt\\
&\leq C\int_{Q_T}\left(|f_1|^2+|f_2|^2\right)e^{-2s\varphi}dxdt+C\int_{\Gamma^+_T}s\lambda\theta\left|\frac{\partial\psi}{\partial \nu}\right|\left(\left|\frac{\partial y_1}{\partial \nu}\right|^2+\left|\frac{\partial y_2}{\partial \nu}\right|^2\right)e^{-2s\varphi}dSdt,
\end{align}
where $\tilde{M}_{11}$, $\tilde{M}_{12}$, $\tilde{M}_{21}$ and $\tilde{M}_{22}$ are defined in \eqref{h1} and \eqref{h2}.
\ep
\begin{proof}
 Using the fact that there exists a $\mu>0$, such that for all $x\in \overline{\Omega}$ and all $\xi=(\xi_1,...,\xi_N)\in\mathbb{R}^N$, $$|\nabla\psi\cdot\xi|^2+\sum_{i,j=1}^N(\partial_i\partial_j\psi(x))\xi_i\xi_j\geq\mu|\xi|^2,$$
 we have
\begin{align}\label{g1}\nonumber
\|\sigma G_1\|+\| G_2\|&\geq \|\sigma M_{11}(u ,v)\|+\|\sigma M_{12}(u ,v)\|+\| M_{21}(u ,v)\|+\| M_{22}(u ,v)\|\\\nonumber
&\q  +C\int_{Q_T}s^3\lambda^4\theta^3\big(|u|^2+|v|^2\big)dxdt+C\int_{Q_T}s\lambda^2\theta\big(|\nabla u|^2+|\nabla v|^2\big)dxdt \\\nonumber
&\q  - \int_{\Gamma^+_T}2s\lambda\theta\frac{\partial\psi}{\partial \nu}\left(\sigma^2\left|\frac{\partial u}{\partial \nu}\right|^2+\left|\frac{\partial v}{\partial \nu}\right|^2\right)dSdt\\
&\q-\varepsilon\int_{Q_T}\left(|f_1|^2+|f_2|^2\right)e^{-2s\varphi}dxdt,
\end{align}
for sufficiently large $s$ and $\lambda$. Here we have used \eqref{d1} and \eqref{e4} in the proof of Proposition 2.1. Noting that $\frac{\partial\psi}{\partial \nu}\leq 0$ on $\Gamma\setminus\Gamma^+$ and $\frac{\partial\psi}{\partial \nu}> 0$ on $\Gamma^+$, we obtain
\begin{align}\label{g2}\nonumber
 &\q\|\sigma M_{11}(u ,v)\|+\|\sigma M_{12}(u ,v)\|+\| M_{21}(u ,v)\|+\| M_{22}(u ,v)\|\\\nonumber
&\q  +C\int_{Q_T}s^3\lambda^4\theta^3\big(|u|^2+|v|^2\big)dxdt+C\int_{Q_T}s\lambda^2\theta\big(|\nabla u|^2+|\nabla v|^2\big)dxdt \\
&\leq \|\sigma G_1\|+\| G_2\|+ \int_{\Gamma^+_T}2s\lambda\theta\left|\frac{\partial\psi}{\partial \nu}\right|\left(\sigma^2\left|\frac{\partial u}{\partial \nu}\right|^2+\left|\frac{\partial v}{\partial \nu}\right|^2\right)dSdt.
\end{align}
Replacing $(u,v)$ by $(e^{-s\varphi}y_1,e^{-s\varphi}y_2)$ in \eqref{g2} yields
\begin{align}\label{p1}\nonumber
&\q  \int_{Q_T}\left(| \tilde{M}_{11}(y_1 ,y_2)|^2+|\tilde{M}_{12}(y_1 ,y_2)|^2+| \tilde{M}_{21}(y_1 ,y_2)|^2+| \tilde{M}_{22}(y_1 ,y_2)|^2\right)e^{-2s\varphi}dxdt\\\nonumber
&\q  +\int_{Q_T}\left[s^3\lambda^4\theta^3\big(|y_1|^2+|y_2|^2\big)+s\lambda^2\theta\big(|\nabla y_1|^2+|\nabla y_2|^2\big)\right]e^{-2s\varphi}dxdt\\
&\leq C\int_{Q_T}\left(|f_1|^2+|f_2|^2\right)e^{-2s\varphi}dxdt+C\int_{\Gamma^+_T}s\lambda\theta\left|\frac{\partial\psi}{\partial \nu}\right|\left(\left|\frac{\partial y_1}{\partial \nu}\right|^2+\left|\frac{\partial y_2}{\partial \nu}\right|^2\right)e^{-2s\varphi}dSdt.
\end{align}
\end{proof}
The following Carleman estimate is a direct consequence of the proposition \ref{D}. Consider the system \eqref{a1}, we have the following result.
\bt\label{E}
Let $\psi$, $\varphi$ and $\theta$ be as in proposition \ref{D}, let $f_1$, $f_2\in L^2(\overline{Q_T})$, and $(y_1,y_2)\in [H^1(0,T;L^2(\Omega))\cap L^2(0,T;H^2(\Omega))]^2$ is the solution of \eqref{a1} with $g_1(x,t)=g_2(x,t)=0$, $a_{ij}\in C^2(\overline{\Omega})$ and satisfy \eqref{ap}.
Then there exist constants $s_0\geq1$, $\lambda_0\geq1$ and $C>0$ such that for all $s>s_0$, $\lambda>\lambda_0$, the next inequality holds:
\begin{align}\label{p11}\nonumber
&\q  \int_{Q_T}\left(| \tilde{M}_{11}(y_1 ,y_2)|^2+|\tilde{M}_{12}(y_1 ,y_2)|^2+| \tilde{M}_{21}(y_1 ,y_2)|^2+| \tilde{M}_{22}(y_1 ,y_2)|^2\right)e^{-2s\varphi}dxdt\\\nonumber
&\q  +\int_{Q_T}\left[s^3\lambda^4\theta^3\big(|y_1|^2+|y_2|^2\big)+s\lambda^2\theta\big(|\nabla y_1|^2+|\nabla y_2|^2\big)\right]e^{-2s\varphi}dxdt\\
&\leq C\int_{Q_T}\left(|f_1|^2+|f_2|^2\right)e^{-2s\varphi}dxdt+C\int_{\Gamma^+_T}s\lambda\theta\left(\left|\frac{\partial y_1}{\partial \nu}\right|^2+\left|\frac{\partial y_2}{\partial \nu}\right|^2\right)e^{-2s\varphi}dSdt.
\end{align}
\et
\subsection{IP (II)}
Based on Carleman inequality \eqref{p11}, we can obtain the following stability result for IP (II).
\begin{theorem}\label{E}
Let $(y_1,y_2)\in [H^1(0,T;L^2(\Omega))\cap L^2(0,T;H^2(\Omega))]^2$ is the solution of \eqref{a1}, $\frac{\partial \tilde{y}_1}{\partial\nu}-\frac{\partial y_1}{\partial\nu},\ \frac{\partial \tilde{y}_2}{\partial\nu}-\frac{\partial y_2}{\partial\nu}\in H^1(0,T;L^2(\Gamma^+))$, and the assumptions $A1-A5$ hold. Then there exists a constant $C>0$ such that
\begin{align}\nonumber
\|\tilde{a}-a\|^2_{L^2(\Omega)}\leq C\left(\|\frac{\partial \tilde{y}_1}{\partial\nu}-\frac{\partial y_1}{\partial\nu}\|^2_{H^1(0,T;L^2(\Gamma^+))}+\parallel\frac{\partial \tilde{y}_2}{\partial\nu}-\frac{\partial y_2}{\partial\nu}\parallel^2_{H^1(0,T;L^2(\Gamma^+))}\right),
\end{align}
where $(y_1,y_2)$, $(\tilde{y}_1,\tilde{y}_2)$ are solutions of \eqref{a1} corresponding to $a$ and $\tilde{a}$, respectively.
\end{theorem}
\begin{proof}
As in the proof of theorem \ref{B}, let $z_1=y_1(a)-\tilde{y}_1(\tilde{a})$, $z_2=y_2(a)-\tilde{y}_2(\tilde{a})$, $f(x)=\tilde{a}-a$ and $R(x,t)=\tilde{y}_1(x,t)$, thus $(z_1,z_2)$ satisfies \eqref{f2}. Then extend the functions $(z_1,z_2)$ on $\Omega\times(-T,T)$ by the formula $\left(z_1(x,t),z_2(x,t)\right)=\left(\overline{z_1(x,-t)},\overline{z_2(x,-t)}\right)$ for $t\in(-T,0)$, and extend $R(x,t)$ by the formula $R(x,t)=\overline{R(x,-t)}$ for $t\in(-T,0)$ if $R(x,0)\in\mathbb{R}$, or extend $R(x,t)$ by the formula $R(x,t)=-\overline{R(x,-t)}$ for $t\in(-T,0)$ if $iR(x,0)\in\mathbb{R}$. Changing $t$ into $t+T$, we consider $(z_1,z_2)$ and $R$ defined on $\Omega\times(0,2T)$, and then we put $(u_1,u_2)=(\partial_t z_1(x,2T-t),\partial_t z_2(x,2T-t))$ so that \eqref{f3} holds, and we can notice that $\eqref{f8}-\eqref{f6}$ are still valid in the current case of boundary observations. By the boundary Carleman inequality \eqref{p11}, we obtain
\begin{align}\nonumber\label{p2}
&\q\int_\Omega e^{-2s\varphi(x,T)}|f(x)|^2dx\\\nonumber
& \leq C\lambda^{-2}s^{-\frac{3}{2}}\left[\int_{0}^{2T}\int_{\Omega}e^{-2s\varphi}|fR_t|^2dxdt+\int_{0}^{2T}\int_{\Gamma^+}s\lambda\theta\left(\left|\frac{\partial u_1}{\partial \nu}\right|^2+\left|\frac{\partial u_2}{\partial \nu}\right|^2\right)e^{-2s\varphi}dSdt    \right]\\
&\leq C\lambda^{-2}s^{-\frac{3}{2}}\int_{\Omega}e^{-2s\varphi(x,T)}|f|^2dx+C\int_{0}^{2T}\int_{\Gamma^+}s^{-\frac{1}{2}}\lambda^{-1}\left(\left|\frac{\partial u_1}{\partial \nu}\right|^2+\left|\frac{\partial u_2}{\partial \nu}\right|^2\right)dSdt.
\end{align}
Therefore, for $s$ and $\lambda$ sufficiently large, we get
\begin{align}\label{p3}
\int_\Omega e^{-2s\varphi(x,T)}|f(x)|^2dx\leq C\int_{0}^{2T}\int_{\Gamma^+}\left(\left|\frac{\partial u_1}{\partial \nu}\right|^2+\left|\frac{\partial u_2}{\partial \nu}\right|^2\right)dSdt.
\end{align}
That is
\begin{align}
\|\tilde{a}-a\|^2_{L^2(\Omega)}\leq C\left(\|\frac{\partial \tilde{y}_1}{\partial\nu}-\frac{\partial y_1}{\partial\nu}\|^2_{H^1(0,T;L^2(\Gamma^+))}+\parallel\frac{\partial \tilde{y}_2}{\partial\nu}-\frac{\partial y_2}{\partial\nu}\parallel^2_{H^1(0,T;L^2(\Gamma^+))}\right).
\end{align}
The proof is completed.
\end{proof}
\section*{Acknowledgement}
The second author thanks the support of the NSFC
(No. 12071061,11971093), the Fundamental
Research Funds for the Central Universities (No.
ZYGX2019J094) and the Science  Strength
Promotion Programme of UESTC.

\end{document}